\documentclass[english, 12pt, final]{amsart}
\def\vint_#1{\mathchoice%
          {\mathop{\kern 0.2em\vrule width 0.6em height 0.69678ex depth -0.58065ex
                  \kern -0.8em \intop}\nolimits_{\kern -0.4em#1}}%
          {\mathop{\kern 0.1em\vrule width 0.5em height 0.69678ex depth -0.60387ex
                  \kern -0.6em \intop}\nolimits_{#1}}%
          {\mathop{\kern 0.1em\vrule width 0.5em height 0.69678ex
              depth -0.60387ex
                  \kern -0.6em \intop}\nolimits_{#1}}%
          {\mathop{\kern 0.1em\vrule width 0.5em height 0.69678ex depth -0.60387ex
                  \kern -0.6em \intop}\nolimits_{#1}}}
\def\vintslides_#1{\mathchoice%
          {\mathop{\kern 0.1em\vrule width 0.5em height 0.697ex depth -0.581ex
                  \kern -0.6em \intop}\nolimits_{\kern -0.4em#1}}%
          {\mathop{\kern 0.1em\vrule width 0.3em height 0.697ex depth -0.604ex
                  \kern -0.4em \intop}\nolimits_{#1}}%
          {\mathop{\kern 0.1em\vrule width 0.3em height 0.697ex depth -0.604ex
                  \kern -0.4em \intop}\nolimits_{#1}}%
          {\mathop{\kern 0.1em\vrule width 0.3em height 0.697ex depth -0.604ex
                  \kern -0.4em \intop}\nolimits_{#1}}}

\usepackage{a4wide}
\usepackage{mathtools}
\usepackage{amsthm}
\usepackage{amssymb}
\usepackage{hyperref}
\usepackage{color}
\mathtoolsset{showonlyrefs}

\usepackage[obeyFinal]{todonotes}

\usepackage{csquotes}

\DeclareMathOperator*{\esssup}{ess\, sup}
\newcommand{\dkakakakakaka}{\mathrm{d}}
\newcommand{\restr}[2]{\ifx#1{\operatorname{restr}}\else {\operatorname{restr}_{#1}^{#2}}\fi}
\renewcommand{\P}{\mathcal{P}}
\newcommand{\A}{\mathcal{A}}
\newcommand{\pr}{\mathrm{P}}
\newcommand{\m}{\mathfrak{m}}
\newcommand{\og}{\mathrm{OptGeo}}
\newcommand{\ent}[1]{\mathrm{Ent_\infty\left( #1 \right)}}
\newcommand{\R}{\mathbb{R}}
\newcommand{\N}{\mathbb{N}}
\newcommand{\geo}[1]{\mathrm{Geo(#1)}}
\newcommand{\abs}[1]{\lvert #1 \rvert}

\renewcommand{\d}{\dkakakakakaka}

\renewcommand{\abs}[1]{\lvert #1\rvert}
\newcommand{\spt}{\mathrm{spt}}
\newcommand{\id}{\mathrm{id}}
\renewcommand{\epsilon}{\varepsilon}

\newtheorem{theorem}{Theorem}[section]
\newtheorem{lemma}[theorem]{Lemma}
\newtheorem{definition}[theorem]{Definition}
\newtheorem{question}{Question}
\begin{document}
\title[Existence of optimal transport maps in very strict $CD(K,\infty)$ -spaces]{Existence of optimal transport maps \\in very strict $CD(K,\infty)$ -spaces}
\author{Timo Schultz}
\address{University of Jyvaskyla \\ Department of Mathematics and Statistics \\
         P.O. Box 35 (MaD) \\
         FI-40014 University of Jyvas\-kyla \\
         Finland}
\email{timo.m.schultz@student.jyu.fi}

\thanks{Author is partially supported by the Academy of Finland. 
}
\date{\today}
\maketitle
\begin{abstract} We introduce a more restrictive version of the strict $CD(K,\infty)$ -condition, the so-called very strict $CD(K,\infty)$ -condition, and show the existence of optimal maps in very strict $CD(K,\infty)$ -spaces despite the possible lack of uniqueness of optimal plans.
\end{abstract}

\section{Introduction}
\todo{Gigli-Rajala-Sturm-viittaus johonkin väliin}
Consider a (complete and separable) metric space $(X,d)$. In the theory of optimal mass transportation probably the most studied problem is the one with the quadratic cost i.e. the minimisation problem
\begin{align}\label{Kantorovich}
\inf \int_{X\times X} d^2(x,y)\,\d\varpi(x,y),
\end{align}
where the infimum is taken over all transport plans $\varpi$ between a given starting measure $\mu_0$ and a final measure $\mu_1$, in other words over all Borel probability measures $\varpi$ on $X\times X$ with marginals $\mu_0$ and $\mu_1$. This problem is interpreted as an optimal mass transportation problem in the following way. The quantity $d^2(x,y)$ tells how much it costs to transport a unit mass from $x$ to $y$, and for a given transport plan $\varpi$, mass from $x$ is to be transported to $y$, if and only if $(x,y)$ belongs to the support of $\varpi$. The total cost of sending all the mass $\mu_0$ to $\mu_1$ according to the plan $\varpi$ is given by the integral \[\int_{X\times X}d^2(x,y)\,\d\varpi(x,y).\]

The above formulation is the so-called Kantorovich formulation of the optimal mass transportation problem, which is a relaxed version of the original Monge formulation of the problem, where instead of the infimum \eqref{Kantorovich} one considers the infimum
\begin{align} \label{Monge} \inf \int_X d^2(x,T(x))\,\d\mu_0(x),
\end{align}
where the infimum is taken over all Borel mappings $T\colon X\to X$ for which $T_\#\mu_0=\mu_1$ i.e. all the Borel maps sending the mass $\mu_0$ to $\mu_1$. One natural and interesting question is whether these two infima agree and when is the optimal plan in \eqref{Kantorovich} given by an optimal map in \eqref{Monge}. More precisely, when is the optimal plan $\varpi$ of the form $\varpi=(\id,T)_\#\mu_0$ for some Borel mapping $T\colon X\to X$. 

The existence of optimal map was proven in the Euclidean setting for absolutely continuous measures by Brenier \cite{Brenier}. Later this result was generalized to the Riemannian framework by McCann \cite{McCann}, and further to some cases of the sub-Riemannian setting by Ambrosio and Rigot \cite{Ambrosio-Rigot}, Agrachev and Lee \cite{Agrachev-Lee}, and by Figalli and Rifford \cite{Figalli-Rifford}. In metric space setting Bertrand proved the existence of optimal map in the so-called Alexandrov spaces \cite{Bertrand}. Under the assumption of non-branching of geodesics, the existence of optimal map was proven for metric measure spaces with Ricci curvature bounded from below i.e. for the spaces satisfying the so-called curvature dimension condition ($CD(K,N)$-condition for short) by Gigli \cite{Gigli} (see Section \ref{Ricci} for the definition of $CD(K,\infty)$-space), and with milder assumptions by Cavalletti and Huesmann \cite{Cavalletti-Huesmann}. 

The non-branching assumption plays a crucial role in both of those proofs leaving the existence of optimal maps in general $CD(K,\infty)$-spaces open. On the other hand, if one considers spaces satisfying only the so-called measure contraction property, which is weaker type of Ricci curvature lower bound condition, the existence of optimal maps may fail as what follows from the example of Ketterer and Rajala in \cite{Ketterer-Rajala}. 

In this paper we go towards understanding the question in general $CD(K,\infty)$-spaces by considering a possibly more restrictive version of $CD(K,\infty)$-property without the non-branching assumption, namely we require the entropy to be convex not only along one optimal geodesic plan but instead along all plans that we get by restricting and weighting a given particular plan (see \ref{vs} for the definition of very strict $CD(K,\infty)$ -space). We prove the existence of optimal maps in very strict $CD(K,\infty)$ -spaces between measures $\mu_0$ and $\mu_1$ that are absolutely continuous with respect to the reference measure $\m$. We actually prove the following stronger statement saying that there exists an optimal dynamical transport plan that is induced by a map. 

\begin{theorem}
\label{thm:optmap}
Let $(X,d, \m)$ be a very strict $CD(K,\infty)$ -space and let $\mu_0, \mu_1\in\P_2(X)$ be absolutely continuous with respect to $\m$. Then there exists a measure $\pi\in\og(\mu_0,\mu_1)$ that is induced by a map i.e. there exists a Borel mapping $T\colon X\to \geo{X}$ so that $\pi=T_\#\mu_0$.
\end{theorem}

In \cite{Rajala2014} Rajala and Sturm (see also \cite{Gigli-Rajala-Sturm}) were able to remove the a priori non-branching assumption by considering a more restrictive version of Ricci curvature lower bounds, namely the strong $CD(K,\infty)$ -property. They introduced the definition of essential non-branching spaces and proved that strong $CD(K,\infty)$ -spaces are essentially non-branching, from which they deduced by using the idea of the proof of Gigli \cite{Gigli} that every optimal plan is given by a map. The result of Rajala and Sturm applies also in the measured Gromov Hausdorff -stable setting of metric measure spaces with Riemannian Ricci curvature bounded from below (the so-called $RCD(K,\infty)$ -spaces, see \cite{Ambrosio-Gigli-Savare, Ambrosio-Gigli-Mondino-Rajala}). While the $CD(K,\infty)$-condition is stable \cite{Sturm, Gigli-Mondino-Savare}, and the strong $CD(K,\infty)$ is not, it remains open whether very strict $CD(K,\infty)$ -, or strict $CD(K,\infty)$ -condition is stable.

\begin{question} Let $(X_i,\d_i,\m_i)_{i\in\N}$ be a sequence of (very) strict $CD(K\infty)$\,-spaces that converge to a metric measure space $(X,\d,\m)$ in suitable sense ( for example in pointed measured Gromov sense). Is $(X,\d,\m)$ a (very) strict $CD(K,\infty)$ -space?
\end{question}

Another related open question is the relation of $CD(K,\infty)$-, strict $CD(K,\infty)$\,-, and very strict $CD(K,\infty)$\,-conditions:

\begin{question}
Does the strict $CD(K,\infty)$ -condition imply very strict $CD(K,\infty)$\,-condition? Does $CD(K,\infty)$ imply (very) strict $CD(K,\infty)$?
\end{question}

With the new notion of essential non-branching introduced by Rajala and Sturm, Cavalletti and Mondino further proved the existence of optimal transport maps for essentially non-branching spaces with the measure contraction property \cite{Cavalletti-Mondino}. Continuing from the work of Cavalletti and Huesmann \cite{Cavalletti-Huesmann} and the work of Cavalletti and Mondino \cite{Cavalletti-Mondino}, Kell proved that in a metric space endowed with qualitatively non-degenerate measure, and therefore especially in spaces satisfying the measure contraction property, the condition of being essentially non-branching is equivalent with having the existence and uniqueness of the optimal transport maps \cite{Kell}.

In the previous results the existence of optimal map is shown by proving that every optimal plan is given by a map and hence also the uniqueness of the plan is guaranteed. In very strict $CD(K,\infty)$ -spaces optimal plans may fail to be unique -- which can be observed by looking for example at the space $\R^n$ equipped with supremum norm -- and therefore this strategy cannot work in our setting. Instead, we should consider one special plan that is given by the defintion of very strict $CD(K,\infty)$ -space. Notice that even though very strict $CD(K,\infty)$ -spaces may fail to be non-branching (even in the sense of essential non-branchingness), still this specific optimal plan does not see any branching geodesics. 

Our proof follows the ideas of Rajala and Sturm in \cite{Rajala2014} and of Gigli in \cite{Gigli}. Instead of proving the existence via the non-branchingness of the optimal plan, we do the proof directly, since it is not clear how to implement the idea of the mixing procedure of \cite{Rajala2014} in the very strict $CD(K,\infty)$ -setting. 

\section{Preliminaries} \label{sec: pre}
Throughout this paper $(X,d,\m)$ is assumed to be a complete and separable metric space endowed with a locally finite Borel measure $\m$. Since we are considering only transportations between absolutely continuous measures, all the Wasserstein geodesics of our concern live in the set of absolutely continuous measures due to the $K$-convexity of the entropy. Thus we may restrict to the case where $X=\spt\,\m$.

By a geodesic we mean a constant speed curve $\gamma\colon [0,1]\to X$ that is length minimizing i.e. $l(\gamma)=d(\gamma_0,\gamma_1)$, where $l$ denotes the length of $\gamma$. We denote by $\geo{X}$ the set of all geodesics of the space $X$ endowed with the supremum metric. 
\subsection{Optimal mass transportation and Wasserstein geodesics}
We denote by $\P(X)$ the space of all Borel probability measures on $X$, and by $\P_2(X)$ the set of all $\mu\in\P(X)$ with finite second moment i.e. those $\mu$ for which we have
\begin{align}\int d^2(x,x_0) \,\d\mu<\infty
\end{align}
for some -- and thus for all -- $x_0\in X$. 

We define the Wasserstein $2$-distance $W_2$ in $\P_2(X)$ as
\begin{align}W_2(\mu,\nu)\coloneqq \left(\inf_{\sigma\in \A(\mu,\nu)}\int_{X\times X}d^2(x,y)\,\d\sigma\right)^{\frac12},
\end{align}
where $\A(\mu,\nu)\coloneqq \left\{\sigma\in\P(X\times X): \pr^1_\#\sigma=\mu,\ \pr^2_\#\sigma=\nu\right\}$ is the set of all admissible plans. The square of the Wasserstein distance is nothing else but the total cost in the mass transportation problem with quadratic cost between the masses $\mu$ and $\nu$. We denote the set of admissible plans realising the above infimum as $\mathrm{Opt(\mu,\nu)}$.

Even though we do not assume the space $(X,d)$ to be geodesic, at the end of Section \ref{sec: pre} we point out that from the definition of very strict $CD(K,\infty)$ -space we actually get that the space $X$ is a length space -- keeping in mind that $X=\spt\,\m$. Since $X$ is a complete and separable metric space with length structure, we also have that the Wasserstein space $(\P_2(X),W_2)$ is a complete and separable length space (see \cite{Villani} and \cite{Lisini}). Furthermore, a curve $t\mapsto\mu_t$ in $\P_2(X)$ is a geodesic if and only if there exists $\pi\in\P(\geo{X})$ such that $(e_t)_\#\pi=\mu_t$ and $(e_0,e_1)_\#\pi\in\mathrm{Opt}(\mu_0,\mu_1)$. Here $e_t\colon \geo{X}\to X$ is the evaluation map $\gamma\mapsto \gamma_t\coloneqq \gamma(t)$. The set of all $\pi\in\P(\geo{X})$ for which $(e_0,e_1)_\#\pi\in\mathrm{Opt}(\mu_0,\mu_1)$ is denoted by $\og(\mu_0,\mu_1)$.

\subsection{Ricci curvature bounded from below} \label{Ricci}

The notion of  synthetic Ricci curvature lower bounds for metric measure spaces were first introduced by Sturm \cite{Sturm} and independently by Lott and Villani \cite{Lott-Villani}. The definitions are based on the connection of Ricci curvature and optimal mass transportation; namely, the convexity properties of suitable entropy functionals along Wasserstein geodesic. For the definition of Ricci curvature lower bounds let us first introduce the \todo{Shannon?} entropy functional $\mathrm{Ent_\infty} \colon \P(X)\to [-\infty,\infty]$ that is defined as 
\begin{align}
\ent{\mu}=\left\{\begin{array}{ll} \int \rho\log\rho \,\d\m &\textrm{, if $\mu\ll \m$ and $\left(\rho\log\rho\right)_+ \in \mathrm{L^1(\m)}$}, \\ \infty & \textrm{otherwise,} \end{array}\right.
\end{align}
where $\rho$ is the density of $\mu$ with respect to $\m$ i.e. $\mu=\rho\m$, and $(\rho\log \rho)_+=\max\{\rho\log\rho,0\}$.

A metric measure space $(X,d,\m)$ is said to have Ricci curvature bounded from below by $K\in\R$, if for every $\mu_0,\mu_1\in \P_2(X)$ absolutely continuous with respect to $\m$, there exists $\pi\in\og(\mu_0,\mu_1)$\todo{määrittele optgeo}\ such that the entropy $\mathrm{Ent_\infty}$ is $K$-convex along $\pi$, that is the inequality
\begin{align}\label{convexity}
\ent{\mu_t}\le (1-t)\ent{\mu_0}+t\ent{\mu_1}-\frac{K}{2}t(1-t)\mathrm{W_2^2(\mu_0,\mu_1)}
\end{align} 
holds for all $t\in[0,1]$, where $\mu_t\coloneqq (e_t)_\#\pi$. Such a space is called a $CD(K,\infty)$-space. If the $K$-convexity holds along every $f\pi$, where $f\colon \geo{X}\to \R$ is any non-negative Borel function for which $\int f\,\d\pi=1$, then the space is called a strict $CD(K,\infty)$ -space (see \cite{Ambrosio-Gigli-Savare}). In this paper a more restrictive version of strict $CD(K,\infty)$ -condition is used, namely the convexity of the entropy is not only required between points $0,$ $t$ and $1$, but also between any points $t_1<t_2<t_3$. To emphasise the difference, we call these spaces \emph{very strict} $CD(K,\infty)$ -spaces.\todo{miten tama oikeasti hoidetaan}

\begin{definition}\label{vs}A metric measure space $(X,d,\m)$ is called a very strict $CD(K,\infty)$ -space, if for every $\mu_0,\mu_1\in\P_2(X)$ absolutely continuous with respect to the reference measure $\m$ there exists $\pi\in\og(\mu_0,\mu_1)$ so that the entropy $\mathrm{Ent_\infty}$ is $K$-convex along $(\restr{t_0}{t_1})_\#(f\pi)$ for all $t_0,t_1\in [0,1]$, $t_0<t_1$, and for all non-negative Borel functions $f\colon\geo{X}\to \R$ with $\int f\,\d\pi=1.$ 
\end{definition}

Here $\restr{t_0}{t_1}\colon\geo{X}\to\geo{X}$ is the restriction map defined as \[\restr{t_0}{t_1}(\gamma)(t)\coloneqq \gamma(t(t_1-t_0)+t_0).\]
It is worth of noticing that due to the Radon-Nikodym theorem, in the above definition one could equivalently require the convexity to hold along $(\restr{t_0}{t_1})_\#\tilde\pi $ for all $\tilde\pi \in \P(X)$ that are absolutely continuos with respect to $\pi$.

As mentioned before, the very strict $CD(K,\infty)$ -condition implies that the space $X$ is a length space: let $x,y\in X$ and define for $\epsilon>0$ the measures $\mu_0\coloneqq \m\lvert_{B(x,\epsilon)}$ and $\m\lvert_{B(y,\epsilon)}$. Let $\pi\in\og(\mu_0,\mu_1)$ which exists by the definition. Now any $\gamma\in\spt\pi$ is a geodesic from $B(x,\epsilon)$ to $B(y,\epsilon)$ and thus the point $\gamma_{\frac{1}{2}}$ is an $\epsilon$-midpoint of $x$ and $y$. Thus by completeness we have that $X$ is a length space \cite{Burago-Burago-Ivanov}.

\section{The Main Theorem}
In the paper \cite{Rajala2014} Rajala and Sturm prove the existence of the optimal map by first proving that every optimal plan $\pi\in\og({X})$ is essentially non-branching and then as a corollary of this they prove that every such $\pi$ is actually given by a map. While the proof for the essential non-branching of the $\pi\in\og(X)$ given by the definition of very strict $CD(K,\infty)$ -space can be carried through with relatively small changes, the proof of the corollary is more problematic. 

In their proof of the existence of optimal map they first divide the original measure $\pi$ into two measures $\pi^1$ and $\pi^2$ that intersect at time $t$, and then construct a new measure $\pi^{mix}$ by mixing these measures $\pi^1$ and $\pi^2$ essentially in the way that at time $t$ you may change from a geodesic in the support of one of the measures to a geodesic in the support of the other. By doing this they end up with a new plan that is still optimal, but due to this mixing the plan is not essentially non-branching anymore. The problem in applying this strategy to our case is that the mixing procedure should be done in such a way that in the end the constructed measure $\pi^{mix}$ is absolutely continuous with respect to the original measure $\pi$, which is something that one should not expect from this kind of mixing. To overcome this obstacle, we prove directly the existence of a map by still using the idea from the proof of the essential non-branching in \cite{Rajala2014}.

In the proof of the existence of optimal map we will use the following two lemmas.
\begin{lemma}[cf. {\cite[Lemma~3.2]{Gigli}}]\label{Gigli}
Let $\pi\in\og(\mu_0,\mu_1)$ be the plan given by the definition of very strict $CD(K,\infty)$. Then for all $\tilde\pi\ll \pi$ with $\m(\{\tilde\rho_0>0\})<\infty$ and $\ent{\tilde\mu_0},\ent{\tilde\mu_1}\in\R$ it holds that 
\begin{align}\m(\{\tilde\rho_0>0\})\le \liminf_{t\to 0} \m(\{\tilde\rho_t>0\}),\end{align}
where $\tilde\rho_t$ is the density of $(e_t)_\#\tilde\pi$ with respect to $\m$.
\end{lemma}

\begin{lemma}\label{simplelemma}
Let $(X,\d)$ be a separable metric space. Then for any $\sigma\in \P(X\times X)$ for which $\sigma\left(\left\{\right(x,x):x\in X\}\right)=0$ there exists $E\subset X$ so that $\sigma(E\times(X\setminus E))>1/5$.
\end{lemma}
The proof of Lemma \ref{Gigli} is the same as the proof of \cite[Lemma~3.2]{Gigli} and the proof of Lemma \ref{simplelemma} can be found in \cite{Rajala2014}.  We will also use the following simple lemma.

\begin{lemma}\label{simplerlemma} Let $\mu_0,\mu_1\in\P_2(X)$ be absolutely continuous with respect to $\m$ with densities $\rho_0$ and $\rho_1$ such that $\ent{\mu_0},\ent{\mu_1}\in\R$ and $\rho_0,\rho_1<C$, and let $\pi\in\og(\mu_0,\mu_1)$ be an optimal plan concentrated on a set of geodesics with length bounded by some constant $L$ such that the convexity inequality \eqref{convexity} holds for all restrictions of $\pi$. Then there exists a constant $M<\infty$ such that for all $t\in[0,1]$ we have $\rho_t\le M$ $\mu_t$-almost everywhere, where $\mu_t\coloneqq (e_t)_\#\pi$ with density $\rho_t$. 
\end{lemma}

\begin{proof} We argue by contradiction. Assume that for all $M>0$ there exists $t_M\in(0,1)$ so that $N_M\coloneqq\mu_{t_M}(\{\rho_{t_M}>M\})>0$. Define \[\hat\pi\coloneqq \pi\lvert_{e_{t_M}^{-1}(\{\rho_{t_M}>M\})}\] and denote $\hat\mu_t\coloneqq (e_t)_\#\hat\pi$ and the density of $\hat\mu_t$ by $\hat\rho_t$. Then we have that
\begin{align}&\phantom{=}\int\frac{\hat\rho_{t_M}}{N_m}\log\frac{\hat\rho_{t_M}}{N_m}\,\d\m-(1-t_M)\int\frac{\hat\rho_{0}}{N_m}\log\frac{\hat\rho_{0}}{N_m}\,\d\m-t_M\int \frac{\hat\rho_{1}}{N_m}\log\frac{\hat\rho_{1}}{N_m}\,\d\m
\\ &\ge \log\frac{M}{N_M}-(1-t_M)\log\frac{C}{N_M}-t_M\log\frac{C}{N_M}=\log M-\log C\to \infty,
\end{align}
when \[M\to\infty.\]
On the other hand by the K-convexity we have for all $M$ that
\begin{align}
&\phantom{=}\int\frac{\hat\rho_{t_M}}{N_m}\log\frac{\hat\rho_{t_M}}{N_m}\,\d\m-(1-t_M)\int\frac{\hat\rho_{0}}{N_m}\log\frac{\hat\rho_{0}}{N_m}\,\d\m-t_M\int \frac{\hat\rho_{1}}{N_m}\log\frac{\hat\rho_{1}}{N_m}\,\d\m
\\ &\le -\frac{K}{2}t_M(1-t_M)\mathrm{W_2^2(\frac{\hat\mu_0}{N_M},\frac{\hat\mu_1}{N_M})}\le \frac{\abs{K}}{2}L^2<\infty.
\end{align}
which is a contradiction. Hence there exists $M$ so that $\rho_t\le M$ for all $t\in[0,1]$. 
\end{proof}
 
\begin{proof}[Proof of Theorem \ref{thm:optmap}]
Let $\mu_0,\mu_1\in \P_2(X)$ be measures that are absolutely continuous with respect to $\m$, and let $\rho_0$ and $\rho_1$ be densities of $\mu_0$ and $\mu_1$ with respect to $\m$. We will prove that the measure $\pi\in\og(\mu_0,\mu_1)$ given by the definition of very strict $CD(K,\infty)$ -space is induced by a map. We will argue by contradiction. Assume that $\pi$ is not induced by a map.  As in \cite{Rajala2014}, we may assume that $\rho_0,\rho_1<C<\infty$ and that the space $X$ is bounded. By $\sigma$-finiteness of $\m$ we may also assume that the $\m$-measure of supports of $\mu_0$ and $\mu_1$ are finite. In particular, by using Jensen's inequality we may also assume that $\ent{\mu_0},\ent{\mu_1}\in \R$.

As in the proof of the essential non-branching of strong $CD(K,\infty)$ -spaces in \cite{Rajala2014}, we want to make the square of the Wasserstein distance $\mathrm{W_2^2(\mu_0,\mu_1)}$ arbitrary small, so that we can basically consider only the convexity of the entropy and forget the $K$ dependent error term in \eqref{convexity}. This is done by the following lemma

\begin{lemma}\label{Localization} If $\pi\in\og(\mu_0,\mu_1)$ is not induced by a map, then for every $k\in\N$ there exists an interval $[i/k,(i+1)/k]$ so that $(\restr{i/k}{(i+1)/k})_\#\pi$ is not induced by a map i.e. $(\restr{i/k}{(i+1)/k})_\#\pi\neq T_\#((e_{i/k})_\#\pi)$.
\end{lemma}
We will postpone the proof of Lemma \ref{Localization} to the end of the paper. Using the above lemma we may restrict the plan $\pi$ so that we have the inequality 
\begin{align}\label{local}L^2\le \frac{\log2}{6\abs{K}+1},
\end{align}
where \[L\coloneqq  \esssup_{\gamma\in\geo{X}} l(\gamma)\]
and the essential supremum is taken with respect to the (restricted) measure $\pi$.

{\bf Step 1:} As in the proof of essential non-branching of strong $CD(K,\infty)$ -spaces, we lift the measure $\pi$ to a measure in $\P(\geo{X}^2)$. Let $\{\pi_x\}$ be a disintegration of the measure $\pi$ with respect to the evaluation map $e_0\colon \geo{X }\to X$. We define $\sigma\in \P(\geo{X}^2)$ by defining the integral of any Borel function $f\colon \geo{X}^2\to [0,\infty]$ with respect to $\sigma$ as
\begin{align}
\int_{\geo{X}^2}f\d\sigma\coloneqq\int_{X}\int_{e_0^{-1}(x)\times e_0^{-1}(x)}f\d (\pi_x\times \pi_x)\,\d\mu_0.
\end{align}
Note that $\sigma$ is well defined since the map $x\mapsto \pi_x\times \pi_x(A)$ is Souslin measurable for every Borel set $A\subset \geo{X}^2$. This can be seen by applying Dynkin's $\pi-\lambda$ theorem to a $\pi$-system $\mathcal{B}(\geo{X})\times \mathcal{B}(\geo{X})$ and a $\lambda$-system $\{A\in \mathcal{B}(\geo{X}^2): x\mapsto \pi_x\times \pi_x(A)\textrm{ is Souslin measurable}\}$.

Since $\pi$ is not induced by a map, there exists $H\subset X$ with positive $\mu_0$-measure so that for any $x\in H$ the measure $\pi_x$ is not a dirac mass. Therefore, there exists $F\subset \geo{X}^2$ for which $\sigma(F)>0$, and it holds that for any $(\gamma^1,\gamma^2)\in F$ we have that  $\gamma^1_0=\gamma^2_0$ and $\gamma^1\neq\gamma^2$. By Lemma $\ref{simplelemma}$ there exists $E\subset \geo{X}$ so that  $\sigma((E\times(X\setminus E))\cap F)>\sigma(F)/5$. Let $\eta>0$ be such that
\begin{align}
\sigma(A)>\frac{1}{10}\sigma(F)
\end{align} 
for the set $A\coloneqq \{(\gamma^1,\gamma^2)\in(E\times(X\setminus E))\cap F: d(\gamma^1,\gamma^2)>\eta \}$. Let $m\in\N$ be large enough so that $\frac{1}{m}<\frac{\eta}{4L}$ and divide the interval $[0,1]$ into $m$ intervals $\{I_j\}_{j=1}^m$ of equal length. Then for every $(\gamma^1,\gamma^2)\in A$ there exist $j\in \{1,\dots, m\}$ and $t\in I_j$ so that $\d(\gamma^1_t,\gamma^2_t)>\eta$ and thus for any $s\in I_j$, since $\abs{t-s}\le\frac1m$, we have that
\begin{align}
\d(\gamma^1_s,\gamma^2_s)\ge \d(\gamma^1_t,\gamma^2_t)-\d(\gamma^1_t,\gamma^1_s)-\d(\gamma^2_t,\gamma^2_s)\ge\eta-2\frac{\eta}{4L}L\ge \frac{\eta}{2}.
\end{align}
Hence there exists $i$ so that $\sigma(A_i)>0$ for $A_i\coloneqq \{(\gamma^1,\gamma^2)\in A : \d(\gamma^1_t,\gamma^2_t)>\frac{\eta}{2}\ \forall t\in I_i\}$. Take a countable partition $P\coloneqq\{Q_j\}$ of $X$ with $ \d(Q_j)<\frac{\eta}{4}$. Then for some $Q\in P  $ we have that $\sigma((e_{S}^{-1}(Q)\times \geo{X})\cap A_i)>0$, where $S\in I_i$ is the mid-point of the interval $I_i$. Now for any $(\gamma^1,\gamma^2)\in (e_{S}^{-1}(Q)\times \geo{X})\cap A_i $ it holds that $\gamma^1_{S}\in Q$ and $\gamma^2_{S}\notin Q$. Define $\tilde\sigma$ as the restriction $\tilde\sigma\coloneqq \sigma\lvert_{(e_{S}^{-1}(Q)\times \geo{X})\cap A_i}$. Furthermore define $\pi^1\coloneqq \pr^1_\#\sigma$ and $\pi^2\coloneqq \pr^2_\#\sigma$. Then $\pi^1,\pi^2\ll \pi$, $\mu^1_0\coloneqq (e_0)_\# \pi^1=(e_0)_\#\pi^2\eqqcolon \mu_0^2$ and $\mu^1_S\bot \mu^2_S$. By restricting the measures $\pi^1$ and $\pi^2$ we may assume that $\rho^i_j<C<\infty$ for some $C>0$ and for all $i\in\{1, 2\}$ and $j\in\{0,1\}$, where $\rho^i_j$ is the density of $(e_j)_\#\mu^i$ with respect to $\m$. By inner regularity of $\mu^1_0$ we may assume that $\mathrm{spt}\,\mu^1_0$ is compact.

{\bf Step 2:} As in the proof of \cite[Theorem~3.3]{Gigli} we will find a time $T\in(0,S)$ so that the intersection of the sets $\{\rho^1_T>0\}$ and $\{\rho^2_T>0\}$ has positive measure. We repeat the argument here. Since $\m$ is locally finite and $\spt\,\mu^1_0$ is compact, there exists a neighbourhood of $\spt\,\mu^1_0$ with finite $\m$-measure. Furthermore there exists an $\epsilon$-neighbourhood $D$ of $\spt\, \mu^1_0$ for which $\m(D)\le \frac{3}{2}\m(\spt\,\mu^1_0)$. By Lemma \ref{Gigli} and by the fact that $L<\infty$ there exists $T\in(0,S)$ so that $\m(\{\rho^i_T>0\})> \frac{3}{4}\m(\{\rho^i_0>0\})$ and $\{\rho^i_T>0\}\subset D $. Hence $\m(\{\rho^1_T>0\}\cap\{\rho^2_T>0\})>0$. Define now \begin{align}\hat\pi^i\coloneqq \frac{1}{\int\min\{\rho^1_T,\rho^2_T\}\d\m}\left(\frac{\min\{\rho_T^1,\rho_T^2\}}{\rho_T^i}\circ e_T\right)\pi^i
\end{align}
for $i\in\{1,2\}$. Then we have that $\hat\pi^1,\hat\pi^2\ll \pi$, $\hat\mu_T^1=\hat\mu_T^2$, $\hat\mu_S^1\bot\hat\mu_S^2$ and $\hat\rho^1_j, \hat\rho^2_j<C<\infty$ for some $C>0$ and for $j\in\{0,1\}$. By Lemma \ref{simplerlemma} there exists $C>0$ so that $\hat\rho^1_t,\rho^2_t<C<\infty$ for every $t\in[0,1]$. 

{\bf Step 3:} Let $T=t_0<t_1<\cdots t_k=S$ be a partition of the interval $[T,S]$ into subintervals of equal length and define $f(i)\coloneqq \int \min\{\hat\rho^1_{t_i},\hat\rho^2_{t_i}\}\d\m$. Then $f(0)=1$ and $f(k)=0$. Therefore there exists $i\in\{0,\dots, k\}$ so that $f(i)-f(i+1)\ge \frac{1}{k}$. Define now $\tilde\pi^1\coloneqq \hat\pi^1\lvert_{e_{t_{i+1}}^{-1}(\{\hat\rho^1_{t_{i+1}}>\hat\rho^2_{t_{i+1}}\})}$ and $\tilde\pi^2\coloneqq \hat\pi^2\lvert_{e_{t_i+1}^{-1}(\{\hat\rho^1_{t_{i+1}}\le\hat\rho^2_{t_{i+1}}\})}$, and further for $j\in\{1,2\}$
\begin{align} \bar\pi^j\coloneqq \left(\frac{\min\{\tilde\rho^1_{t_i}, \tilde\rho^2_{t_i}\}}{\tilde\rho^j_{t_i}}\circ e_{t_i}\right)\tilde\pi^j.
\end{align}
Then $\bar\mu^1_{t_{i+1}}(\{\hat\rho^1_{t_{i+1}}>\hat\rho^2_{t_{i+1}}\})=\bar\mu^1_{t_{i+1}}(X)$ and $\bar\mu^2_{t_{i+1}}(\{\hat\rho^1_{t_{i+1}}\le\hat\rho^2_{t_{i+1}}\})=\bar\mu^2_{t_{i+1}}(X)$. Thus we have that $\bar\mu^1_{t_{i+1}}\bot \bar\mu^2_{t_{i+1}}$. By definition we also have that $\bar\mu^1_{t_i}=\bar\mu^2_{t_{i}}$. Let us prove that \todo{Here $\hat\mu(X)=1$}
\begin{align}\label{big} \bar\mu^1_{t_{i}}(X)\ge \frac{1}{k}.
\end{align}
By definition we have that $\bar\rho_{t_i}^j\le\tilde\rho_{t_i}^j\le \hat\rho_{t_i}^j$ for $j\in\{1,2\}$. Also we have that 
\begin{align}
\tilde\mu_{t_i}^1(X)+\tilde\mu_{t_i}^2(X)&=\tilde\mu_{t_{i+1}}^1(X)+\tilde\mu_{t_{i+1}}^2(X)
\\&=\int\hat\rho_{t_{i+1}}^1+\hat\rho_{t_{i+1}}^2-\min\{\hat\rho_{t_{i+1}}^1,\hat\rho_{t_{i+1}}^2\}\d\m
\\&= 2-f(i+1).
\end{align}
Therefore we get that
\begin{align}\bar\mu_{t_i}^1(X)&=\int \min\{\tilde\rho_{t_i}^1,\tilde\rho_{t_i}^2\}\d\m=\tilde\mu_{t_i}^1(X)+\tilde\mu_{t_i}^2(X)-\int \max\{\tilde\rho_{t_i}^1,\tilde\rho_{t_i}^2\}\d\m
\\ &\ge2-f(i+1)-\int\max\{\hat\rho_{t_i}^1,\hat\rho_{t_i}^2\}=f(i)-f(i+1)\ge \frac{1}{k}.
\end{align}

{\bf{Final step:}} Now we are ready to arrive to a contradiction with the convexity of the entropy. As in \cite{Rajala2014} we consider three measures $\bar\pi^1/M$, $\bar\pi^2/M$ and $(\bar\pi^1+\bar\pi^2)/(2M)$ along which the entropy is $K$-convex. Here $M\coloneqq \bar\pi^i(\geo{X})$. For these measures we have that $\bar\rho^i_t<C$ for all $t\in [0,1]$ and $i\in \{1,2\}$, $M>1/k$, $\bar\mu^1_{t_i}=\bar\mu^2_{t_i}$ and $\bar\mu^1_{t_{i+1}}\bot \bar\mu^2_{t_{i+1}}$. From these facts we get that
\begin{align}
&\phantom{=}\int \frac{\bar\rho^1_{t_i}}{M}\log \frac{\bar\rho^1_{t_i}}{M}\d\m \\
&\le\frac{1}{kt_{i+1}}\int \frac{\bar\rho^1_0+\bar\rho^2_0}{2M}\log\frac{\bar\rho^1_0+\bar\rho^2_0}{2M}\d\m+\frac{t_i}{t_{i+1}}\int \frac{\bar\rho^1_{t_{i+1}}+\bar\rho^2_{t_{i+1}}}{2M}\log\frac{\bar\rho^1_{t_{i+1}}+\bar\rho^2_{t_{i+1}}}{2M}\d\m\\
&\phantom{=}+\frac{\abs{K}}{2}\frac{t_{i}}{t_{i+1}}\mathrm{W_2^2(\bar\mu_0,\bar\mu_1)} \\
&\le \frac{\log\frac{C}{M}}{kt_{i+1}}-\frac{t_i}{t_{i+1}}\log2+\frac{t_i}{2t_{i+1}}\left(\int \frac{\bar\rho^1_{t_{i+1}}}{M}\log \frac{\bar\rho^1_{t_{i+1}}}{M}\d\m+\int \frac{\bar\rho^2_{t_{i+1}}}{M}\log \frac{\bar\rho^2_{t_{i+1}}}{M}\d\m\right)\\ 
&\phantom{=}+\frac{\abs{K}}{2}\frac{t_{i}}{t_{i+1}}L^2 \\
&\le \frac{\log\frac{C}{M}}{kt_{i+1}}-\frac{t_i}{t_{i+1}}\log2+\frac{t_i}{2t_{i+1}}\left[\frac{2(1-t_{i+1})}{1-t_i}\int \frac{\bar\rho^1_{t_i}}{M}\log \frac{\bar\rho^1_{t_i}}{M}\d\m\right. \\
&\phantom{=}+\left.\frac{1}{k(1-t_i)}\left(\int \frac{\bar\rho^1_{1}}{M}\log \frac{\bar\rho^1_{1}}{M}\d\m+\int \frac{\bar\rho^2_{1}}{M}\log \frac{\bar\rho^2_{1}}{M}\d\m\right)\right] +\abs{K}\frac{t_{i}}{t_{i+1}}L^2\\
&\le \left(\frac{1}{kt_{i+1}}+\frac{t_i}{t_{i+1}k(1-t_{i+1})}\right)\log\frac{C}{M}-\frac{5}{6}\frac{t_i}{t_{i+1}}\log2+\frac{t_i(1-t_{i+1})}{t_{i+1}(1-t_i)}\int \frac{\bar\rho^1_{t_i}}{M}\log \frac{\bar\rho^1_{t_i}}{M}\d\m
\end{align}
from which we obtain
\begin{align}
\frac{1}{kt_{i+1}(1-t_i)}\int \frac{\bar\rho^1_{t_i}}{M}\log \frac{\bar\rho^1_{t_i}}{M}\d\m\le \frac{1}{kt_{i+1}(1-t_i)}\log\frac{C}{M}-\frac{5}{6}\frac{t_i}{t_{i+1}}\log2
\end{align}
and further
\begin{align}
\int \frac{\bar\rho^1_{t_i}}{M}\log \frac{\bar\rho^1_{t_i}}{M}\d\m\le\log\frac{C}{M}-\frac{5k}{6}(1-t_i)t_i\log2.
\end{align}
Choosing $k$ large enough so that $T,(1-S)^2\ge \frac{1}{k}$, and using the above inequality together with the convexity of the entropy along $\bar\pi^1$ we get that
\begin{align}
&\phantom{=}\int \frac{\bar\rho^1_{t_{i+1}}}{M}\log \frac{\bar\rho^1_{t_{i+1}}}{M}\d\m\le \frac{t_{i+1}-t_i}{1-t_i}\log \frac{C}{M}+\frac{1-t_{i+1}}{1-t_i}\int\frac{\rho^1_{t_i}}{M}\log\frac{\rho^1_{t_i}}{M}\\ 
&\phantom{=}+\frac{\abs{K}}{2}\frac{(1-t_{i+1})}{k(1-t_i)^2}L^2\\
&\le \left(\frac{t_{i+1}-t_i}{1-t_i}+\frac{1-t_{i+1}}{1-t_{i}}\right)\log\frac{C}{M}-\frac{5k}{6}t_i(1-t_{i+1})\log2\\
&\phantom{=}+\frac{\abs{K}}{2}kt_i(1-t_{i+1})L^2 \\
&\le \log\frac{C}{M}-\frac{k}{3}t_i(1-t_{i+1})\log2
\end{align}
from which by \eqref{big} and Jensen's inequality we obtain \todo{What if $\m(X)=\infty$?}
\begin{align}
&\phantom{=}\log\frac{1}{k\m(X)}-\log M\le \int \frac{\bar\rho^1_{t_{i+1}}}{M}\log\bar\rho^1_{t_{i+1}}\d\m-\log M=\int \frac{\bar\rho^1_{t_{i+1}}}{M}\log\frac{\bar\rho^1_{t_{i+1}}}{M}\d\m\\
&\le \log\frac{C}{M}-\frac{k}{3}t_i(1-t_{i+1})\log2.
\end{align}
Hence we arrive to a contradiction
\begin{align}
\frac{3}{k}(\log\frac{1}{k\m(X)}-\log C)\le -t_i(1-t_{i+1})\log2\le-T(1-S)\log2
\end{align}
since the left hand side goes to zero when $k$ goes to infinity while the right hand side is negative and bounded away from zero.
\end{proof}

\begin{proof}[Proof of Lemma \ref{Localization}]
By an induction argument it suffices to prove that if $\pi\in \og(\mu_0,\mu_1)$ is not induced by a map, then  for each $t\in(0,1)$ we have that $(\restr{0}{t})_\#\pi$ or $(\restr t1)_\#\pi$ is not induced by a map. %For notational simplicity we assume that $t=\frac{1}{2}$. The argument for general $t$ goes exactly the same way. \todo{miten ton concatenationin tai vaihtoehtosesti oletuksen  $t=1/2$ muotoilee järkevästi}

Suppose that there is $t\in(0,1)$ so that both $\pi^1\coloneqq (\restr0t)_\#\pi$ and $\pi^2\coloneqq (\restr t1)_\#\pi$ are induced by maps $T^1$ and $T^2$ respectively. Denote $\{\pi_x\}$, $\{\pi^1_x\}$, $\{\pi^2_x\}$ the disintegrations of $\pi$, $\pi^1$ and $\pi^2$ with respect to $e_0$. It suffices to prove that $\spt\,\pi_x$ is a singleton for $\mu_0$-a.e. $x\in X$. We will do this by proving that actually 
\begin{align}\spt\,\pi_x=\spt\,\pi^1_x* \spt\,\pi^2_{\hat T(x)},
\end{align}
where $\hat T\coloneqq e_1\circ T^1$ is the optimal map from $\mu_0$ to $\mu_{t}$, and $*$ is the concatenation of paths. 

We begin by observing that for $\mu_0$-a.e. $x$ we have, by the definition of disintegration,\todo{or uniqueness?}\ that $\pi^1_x=(\restr0t)_\#\pi_x$. In particular we have that $\spt\,\pi^1_x=\restr{0}{t}\,\spt\,\pi_x$. Thus, since both $\spt\, \pi^1_x$ and $\spt\,\pi^2_{\hat T(x)}$ are singletons for $\mu_0$-a.e. $x\in X$, it suffices to prove that $\pi_x((\restr t1)^{-1}\spt\, \pi^2_{\hat T(x)})=1$ for $\mu_0$-a.e. $x\in X$. Suppose this is not the case. Then the set 
\begin{align}E\coloneqq \{x:\pi_x((\restr t1)^{-1}\spt\, \pi^2_{\hat T(x)})<1\}\end{align} has positive $\mu_0$-measure. Consider the set $F\coloneqq \cup_{x\in E}\spt\, \pi^2_{\hat T(x)}$. Then
\begin{align}
\pi^2(F)&=\pi\left((\restr t1)^{-1}F\right)=\int \pi_x\left(\bigcup_{y\in E} (\restr t1)^{-1}\spt\, \pi^2_{\hat T(y)}\right)\,\d\mu_0(x)\\
&=\int_E \pi_x\left(\bigcup_{y\in E} (\restr t1)^{-1}\spt\, \pi^2_{\hat T(y)}\right)\,\d\mu_0(x)+\int_{X\setminus E} \pi_x\left(\bigcup_{y\in E} (\restr t1)^{-1}\spt\, \pi^2_{\hat T(y)}\right)\,\d\mu_0(x) \\
&=\int_E \pi_x\left((\restr t1)^{-1}\spt\, \pi^2_{\hat T(x)}\right)\,\d\mu_0(x)+\int_{X\setminus E} \pi_x\left(\bigcup_{y\in E} (\restr t1)^{-1}\spt\, \pi^2_{\hat T(y)}\right)\,\d\mu_0(x) \\
&< \mu_0(E)+\int_{X\setminus E}\pi_x\left(\bigcup_{y\in E}(\restr t1)^{-1}\spt\,\pi^2_{\hat T(y)}\right)\,\d\mu_0(x) \\
&=\mu_0(E)+\mu_0(\{x\in X\setminus E: \exists y\in E\textrm{ for which }\hat T(x)=\hat T(y) \})\\
&=\int_E \pi^2_{\hat T(x)}\left(\bigcup_{y\in E}\spt\, \pi^2_{\hat T(y)}\right) \,\d\mu_0(x)+ \int_{X\setminus E} \pi^2_{\hat T(x)}\left(\bigcup_{y\in E}\spt\, \pi^2_{\hat T(y)}\right) \,\d\mu_0(x) \\
&=\pi^2(F)
\end{align}
which is a contradiction. Thus we have proven the lemma.
\end{proof}

\bibliographystyle{amsplain}
\bibliography{Lahteetteethal}
\end{document}